\documentclass[12pt, a4 paper]{article}
\usepackage{graphicx} 
\usepackage{xcolor}
\usepackage[utf8]{inputenc}
\usepackage[margin=1in]{geometry}
\usepackage{amssymb}
\usepackage{amsmath}
\usepackage{amsthm}
\usepackage[ruled,vlined]{algorithm2e}
\usepackage{graphicx}
\usepackage{hyperref}
\usepackage{mathtools}
\usepackage{nccmath}
\usepackage{setspace}
\usepackage{float}
\usepackage{caption}
\usepackage{subcaption}
\usepackage{enumerate}
\usepackage{setspace}
\usepackage{hyperref}
\theoremstyle{definition}
\newtheorem{definition}{Definition}[section]
\newtheorem{theorem}{Theorem}[section]
\newtheorem{lemma}{Lemma}[section]

\newtheorem{corollary}{Corollary}[section]
\newtheorem{example}{Example}[section]
\newtheorem{remark}{Remark}[section]
\title{Information retrieval in big data using cognitive approaches}
\author{Santanu Acharjee$^{1,2}$ and Ripunjoy Choudhury$^{1,3}$\\
$^{1}$Department of Mathematics\\
Gauhati University\\
Guwahati-781014, Assam, India\\
$^{2}$Indian National Young Academy of Science \\
2, Bahadur Shah Zafar Marg\\ New Delhi-110 002, India\\
$^{1,3}$Department of Mathematics\\ Kamrup Polytechnic\\ Baihata Chariali-781381, Assam, India\\
e-mails: $^{1,2}$sacharjee326@gmail.com,$^{1,3}$ripunjoy07@gmail.com\\
Corresponding author: Santanu Acharjee}
\date{}
\begin{document}
\maketitle
\noindent
 {\bf Abstract:} Due to the exponential growth of big data in this digital era, an advanced method for effective information retrieval becomes essential. The basic objective of this paper is to propose a topology-based method for cognitive information retrieval (CIR) in  big data environments. By using concepts such as cognitive similarity distances, metric spaces, retrieval topologies, etc., this paper aims to propose the semantic alignment between user queries and document repositories. The paper also extends this approach to incorporate logical
connectives in cognitive information retrieval.\\

\vskip 0.2cm
\noindent
{\bf 2020 AMS Classifications:} 94A16; 68T09; 68P01; 68T27; 62R40;30L15 .\\

\section{Introduction:}Indexing and searching are the two primary components of information retrieval (IR), and more specifically document retrieval. Indexing is a process used to store a document in a database, and searching involves using queries. Searching explains how users communicate their need for information or a subject they'd want to learn more about.  The internet is not the start of the lengthy history of information retrieval. Only in the past ten years web search engines have evolved into search, and pervasiveness has been incorporated into desktop and mobile operating systems. Information retrieval (IR) technologies were used in intelligence and commercial applications in the 1960s, before search engines became widely used in daily life \cite{1}. Inspired by remarkable inventions in the first half of the 20th century, the first computer-based search systems were constructed in the late 1940s \cite{1}. The capabilities of the retrieval  system increased as the processor speed and storage capacity increased, as is the case with many computer technologies\cite{1}. Information that is relevant to a user's inquiry is found by an IR system. Typically, an IR system looks through collections of semi-structured or unstructured data. When a collection grows too large for conventional cataloging methods to handle, an IR system becomes necessary\cite{1}.\\

The advancement of the World Wide Web has produced tremendous volumes of data quickly and many strategies are being used to handle this massive growth of data. According to a recent survey \cite{2}, global data production, capture, copying and consumption were expected to increase rapidly, with a projected total of 64.2 zettabytes in 2020. It is expected that after five more years of growth, the amount of data created worldwide will exceed 180 zettabytes in 2025 \cite{2}. The quantity of data reached a record level in 2020 \cite{2}. The increased demand for the COVID-19 pandemic, which resulted in more people working and learning at home and using home entertainment alternatives more frequently, led the expansion to be greater than expected \cite{2}. The exponential growth of data in the digital era has led to the advent of big data, characterized by its  volume, variety, velocity, veracity and value.\\

The term `big data' itself signifies a large volume of data which is growing at an unprecedented rate \cite{3}.  In the upcoming years, the data size is anticipated to grow from petabytes to exabytes \cite{3}. As technology advances, the amount of data is increasing rapidly, and appropriate tools should be available to manage this enormous data explosion \cite{3}. This surge necessitates advanced methods for efficient information retrieval (IR), ensuring relevant and accurate data extraction from massive datasets.\\

Traditional information retrieval methods can be classified into four parts: Boolean, vector space, probabilistic and inference network model. One may refer to \cite{4,5} to know more about each of these methods. Along with these methods, theoretical methods of studying IR based on topology were also developed by Cater \cite{6} in 1986. 
After that, many researchers were involved in developing these new topology-based methods. Egghe \cite{7},  Egghe and  Rousseau \cite{8}, considered (DS, QS, $Sim, \tau)$ to be a retrieval system consisting of a document space DS, a query space QS, and a function $Sim$, expressing the similarity between a document and a query. The topology $\tau$ is generated by the set of such documents whose similarity with queries is less than some threshold value. In this regard, one may refer to \cite{7,8}. In recent times, some studies related to the topology of the IR system and its applications to many fields can be seen like LLM\cite{9}, some simulated information ecosystem\cite{10} and 3DIR\cite{11}. It ensures that how much the topology-based method of IR is relevant in today's world too. Thus, the uses of topology-based methods for IR can be seen as an efficient tool in this big data era. Though the topology-based methods invented by Carter\cite{6} and further largely used by Egghe and Rousseau\cite{8} that time, yet to cope with it in this era, need some further modification in the methods. Their methods were useful to retrieve documents on the basis of similarity functions, like cosine similarity, Jaccard similarity, etc., but they fail to understand the cognitive perspective of users' queries. Thus, a new and an upgraded version of the topology-based IR method using cognitive theory is developed in this paper.

\section{Preliminaries:}

In this section, we procure some established definitions and results. They will be used in our next section.
\begin{definition}\cite{7}
	Let DS be the set (space) of all documents and QS a set of queries and $Sim(., Q)$ be a function of DS in W,
	measuring the similarity $Sim(D, Q)$ between a document $D$ and a query $Q$. Most commonly, its values are in $\mathbb{R}^{+}$ or even [0,1], but this restriction is not necessary. Then, the retrieval topology $\tau$ on DS is defined to be the topology generated by the sets (as a subbasis) $R(Q,r) = \{\, D\in $DS$ \mid $ $Sim$ $(D, Q) > r\,\},$ for $ Q\in$ QS and $r\in \mathbb R.$ 
\end{definition} 
\begin{definition}\cite{7}
	In the system (DS, QS,$ sim, \tau )$, the set
	ret$_{\tau}(Q) = \{\, R(Q, r)\mid  r\in \mathbb R\,\}$ 
	is called the set of retrievals of $Q \in$ QS.
\end{definition}
\begin{definition}\cite{12} A non-empty set $X$ with a map $d : X \times X \longrightarrow [0,\infty) $ is called a metric space if the map $d$ has the following properties:
	\begin{enumerate}
		\setlength{\itemsep}{0em}
		\parskip=0pt
		
		\item $d(x, y)\geq 0$ $\forall x, y \in X,$
		\item $d(x, y) = 0$ if and only if $x = y$,
		\item $d(x, y)= d(y, x)$, $ \forall x, y \in X$,
		\item $d(x, y) \leq d(x,z) + d(z, y)$, $ \forall x, y,z,\in X.$
	\end{enumerate}
	
	\end{definition}
	In \cite{13}, Acharjee and Gogoi introduced some important notions related to cognition, which are stated below:

\begin{definition}\cite{13}
	Let $\mathcal{C}$ be a cognitive consequence space. The cognitive similarity distance on $\mathcal{C}$ can be defined as a map $Cog: \mathcal{C}\times \mathcal{C} \rightarrow [0,1],$ which satisfies the following properties:
	\begin{enumerate}
	
		\item $Cog (x,y)\geq 0 ,$
		\item $Cog (x,y) = 0 \iff x \approx y,$
		\item $Cog(x,y)= Cog(y,x),$
		\item If $x\approx z,$ then $Cog(x,y)=Cog(z,y),$
		\item $Cog(x,y)\leq Cog(x,z)+ Cog(z,y)$, where $x,y,z\in C.$
	\end{enumerate}
\end{definition}

In \cite{13},  Acharjee and Gogoi further mentioned that $x \approx y$ indicates that the two thoughts $x$ and $y$ are semantically similar or cognitively similar or identical. In short, we say that $x$ and $y$ cognitively coincide. The range is considered [0, 1] since most of the similarity measures are calculated between 0 and 1. For example, the Jaccard similarity measure\cite{14}, cosine similarity\cite{15},etc. are calculated in the interval [0, 1].

\section{Cognitive retrieval system in big data:}
In the era of big data analytics, textual data are growing exponentially, and so some sophisticated techniques are required to retrieve information effectively. Traditional information retrieval (IR) systems mostly depend on keyword matching \cite{16}, which frequently produces poor search results since the algorithms cannot understand the semantics or context of the queries in a better way. By combining cognitive computing and artificial intelligence (AI) technologies, cognitive information retrieval (CIR) seeks to overcome these shortcomings by improving the understanding and retrieval of information based on the context and intent of user queries \cite{17}. Before going to formulate mathematical method for cognitive similarity between queries and documents, it is necessary to understand more about semantic similarity. \\

In essence, semantic similarity is a measure that determines how similar two concepts are to one another based on the likelihood that they mean the same thing \cite{18}. It encompasses a range of linguistic and cognitive elements, including context, related concepts, and synonyms, which enable systems to deduce connections between different words and sentences. In cognitive information retrieval (CIR), semantic similarity is essential as it aims to find materials that match a user's context and cognitive demands \cite{18}. This allows the algorithms to recognize the meanings and complexity of the contexts, producing more appropriate search results. For example, `car' and `automobile' are semantically similar, making it useful to retrieve information that uses either term. For more details about semantic similarity and its importance in information retrieval, one may refer to \cite{18}. In the following part, we are going to give some examples showing how Google algorithms retrieve information on the basis of cognitive similarity or semantic similarity.
Suppose a user searches the following in Google: \\ 

\textbf{Example 1:}\\
\textbf{Query}: `Best places to visit in New York',\\ \textbf{Document}: Coolest things to see and explore in New York City, \\
\textbf{Explanation:} Although the actual wording is altered, both the query and the documents are about notable locations to visit in New York City.\\

\textbf{Example 2:}\\
\textbf{Query}: `Benefits of a plant-based diet',\\
\textbf{Document}: Health benefits of a vegan diet.\\
\textbf {Explanation:} Here, `plant-based diet' and `vegan diet' are two semantically related words. So, the Google algorithm retrieves this document, although the actual words are different. \\

\textbf{Example 3:}\\
\textbf{Query}: `Effective ways to lose weight',\\
\textbf{Documents}: Science-backed tips to lose weight fast and sustainably\\
\textbf {Explanation:} The query and the document both refer to successful methods for reducing body weight.\\

The aforementioned examples sufficiently demonstrate cognitive similarities between queries and retrieved information, as they cognitively coincide. Although Egghe \cite{7} discusses various notions regarding the topological foundation in information retrieval, but any article addressing the topological foundations of cognitive information retrieval in the context of big data hardly can be found. Therefore, in this section, we introduce a upgarded topology based method of  information retrieval system using cognitive theory in big data analytics.

It is evident that we most often express our thoughts in words. Therefore, whenever we make a query, it is simply a word or combination of words representing our thoughts. In this paper, we write a query $q$ as $q=x_{1}\oplus x_{2}\oplus...\oplus x_{k}$, where $x_{i}, i=1,2,..,k$ represents words and $\oplus$ denotes the concatenation between words. \\

In Example 2, we observe that the word `Benefits' in the retrieved document matches `Health benefits.' These terms are not semantically similar. Because `Benefits' is a broader term that encompasses various advantages or gains, while `Health benefits' is more specific, highlighting advantages related to well-being and health. As discussed by Acharjee and Gogoi \cite{13}, cognitive similarity distances between thoughts or concepts can be considered. Therefore, in this case, we can infer some cognitive similarity distances between the words `Benefits' and `Health benefits'. In the following, we defined some new notions related to cognitive similarity:\\

\begin{definition}
	Let $W$ be the set of all words. Consider $A$ and $B$ to be two finite subsets of $W$. Then, the subsets of words $A$ and $B$ are cognitively similar sets if for each $x\in A$, there exists $y\in B$  such that $Cog(x,y)=0$, and we denote it as $A\approx B$. \end{definition}
\begin{example}
	Let $A=\{\,$road, way$\,\}$ and $B=\{\,$path, lane$\,\}$. Then, it is clear that each word in $A$ is cognitively similar to each word of $B$. Hence, $A\approx B$. \end{example}

In \cite{13}, Acharjee and Gogoi  considered  cognitive similarity distance between thoughts. So, using their results, we develop cognitive similarity distance from a word to a set of words $A$ as $Cog^{*}(z,A)=\sum_{x\in A}Cog(z,x)$ . Thus, we develop a new notion of cognitive distance between two sets of words below: 
\begin{definition}
Let $W$ be the set of words. Consider $A$ and $B$ to be two finite subsets of $W$. We define a function $Cd: 2^{W} \times 2^{W} \longrightarrow [0,1] $ such that, $$Cd(A,B) = \left\{ \begin{array}{rcl}
	0 & \mbox{,} & if  A \approx B, \\\sup\{\,|Cog^{*}(z,A)-Cog^{*}(z,B)|: z\in W\,\} & \mbox{,}
		& otherwise.  \\
	
	\end{array}\right.$$   We call the function `$Cd$' as cognitive distance of words.  
\end{definition}

\begin{lemma}
	The cognitive distance for words defined in definition 3.2 satisfies the following properties:
    
    \begin{enumerate}[(i)] 
		\item $Cd(A,B)\geq 0$,
		\item $Cd(A,B)=Cd(B,A)$,
	\item  $Cd(A,B)\leq Cd(A,C)+Cd(C,B)$, where  $A,B,$ and $C$ are finite subsets of $W$.
	
	\end{enumerate}

	\end{lemma}

\begin{proof}
Let $A, B,$ and $C$ be any finite subsets of $W$.
\begin{enumerate}[(i)]
	\item If $A\approx B$, then $Cd(A,B)=0$. Otherwise,  $\sup\{\,|Cog^{*}(z,A)-Cog^{*}(z,B)|: z\in W\,\}\geq 0.$ Thus, $Cd(A,B)\geq 0.$
	\item  We have, $\sup\{\,|Cog^{*}(z,A)-Cog^{*}(z,B)|: z\in W\,\}=\sup\{\,|Cog^{*}(z,A)-Cog^{*}(z,B)|: z\in W\,\}$. Thus, $Cd(A,B)=Cd(B,A).$
	\item For any $z\in W$, we have $|Cog^{*}(z,A)-Cog^{*}(z,B)|=|Cog^{*}(z,A)-Cog^{*}(z,C)+Cog^{*}(z,C)-Cog^{*}(z, B)
    |\leq|Cog^{*}(z,A)-Cog^{*}(z,C)| + |Cog^{*}(z,C)-Cog^{*}(z,B)|$. Hence, we get that $Cd(A,B)\leq Cd(A,C)+Cd(C,B)$.
\end{enumerate} 
 \end{proof}
 \begin{remark}
     From definition 3.2 and lemma 3.1, it is clear that function $`Cd'$ is a pseudo metric defined on $2^{W}.$
 \end{remark}

\begin{lemma}
Let $W$ be the set of words. The relation $\approx$ on finite subsets of $W$ satisfies that, for any finite subsets $A, B, C$ of $W$, if $A\approx C$ and $A\approx B$, then $Cd(B,C)=0$.

\begin{proof}
	
    It is given that $A\approx C$ and $A\approx B$. So, we have $Cd(A,C)=0$ and $Cd(A,B)=0$. Now, from lemma 3.1, we have $Cd(B,C)\leq Cd(B,A) + Cd(A,C)$. This implies that $Cd(B,C)\leq Cd(A,B)+ Cd(A,C)$, which becomes $Cd(B,C)\leq 0$. But, by lemma 3.1, we also have that $Cd(B,C)\geq 0.$ Hence, $Cd(B,C)=0.$

\end{proof}
\end{lemma}

It is evident that, when users search for some specific information in a database, they submit a query which is nothing but a collection of words which represents their thoughts, and consequently, they retrieve  multi-modal data containing image, video, text etc. But there are some literature \cite{19,20,21} stating that multi modal data can be converted to text format, which is nothing but a collection of words. So, we can say that queries and retrieved documents are nothing but some collections of words. Thus, as stated in the previous part, the cognitive distance, `$Cd$' can also be considered between queries, documents and both.\\

When a query is submitted to a retrieval system, the algorithm begins processing the query within the database. Processing the entire database may decrease retrieval efficiency and slow down the response time \cite{22}. It is evident that in big data, processing happens in a distributed manner. Thus, in big data, to process a query, partitioning the document set becomes crucial to overcome these challenges. In large-scale data environments, it is impractical to process every document for every query. Partitioning based on cognitive distance allows systems to handle vast data sets more effectively by considering only the subsets of documents likely to be relevant. Thus, we divide the set of documents into some partitions [$d$], where each partition contains documents that are cognitively similar or semantically similar or identical to $d$, i.e., if $d_{1}\in [d]$, then $d\approx d_{1}$ and vice versa. Let $\mathcal{P}=\{\,[d]\mid d\in  DS\,\}$ be the collection of such partitions.
\begin{lemma}
Let $d$ and $d^{/}$ be any two documents of DS such that $d\approx d^{/}$. Then, we have that $[d]= [d^{/}]. $
\end{lemma}

\begin{proof}
Let $m\in [d]$. Then, $m\approx d.$ Now, $d\approx d^{/}$ so, $Cd(m,d)=0$ and $Cd(d,d^{/})=0$. Thus, $Cd(m, d^{/})\leq Cd(m,d)+Cd(d,d^{/}).$ This implies that $Cd(m,d^{/})=0$ and so, $m\approx d^{/}.$ Thus, $m\in [d^{/}]$ and so, $[d]\subseteq [d^{/}].$ In similar manner, it can be shown that, $[d]\supseteq[d^{/}].$ Hence, $[d]=[d^{/}].$ 
\end{proof}
\begin{remark}
Here, it can be notice that $[d], [d^{/}]$ are either equal or completely disjoint. Let us assume that $[d]\neq [d^{/}]$  and  there exist some $m\in [d]\cap [d^{/}]$. Then, $m\approx d, m\approx d^{/}$.Thus,  by lemma 3.2, we have $Cd(d,d^{/})=0$. It gives that $d\approx d^{/}$. Thus, by lemma 3.3, $[d]=[d^{/}]$ and it contradicts our assumption. Hence,  $[d], [d^{/}]$ are either equal or completely disjoint.\\
	
\end{remark}

Let us consider two different partitions of document sets, one containing research papers on machine learning, natural language processing, and artificial intelligence, and the other containing research papers on quantum computing, cryptography, and information security; then the measure for cognitive similarity between these two sets using an appropriate distance formula will be helpful in the information retrieval process. Thus, let us define the distance formula for partitions.\\

\begin{definition}
Let $B$ be the universe of big data, and DS be the finite document space of $B$, where DS $\subseteq B$. We define a cognitive distance $Cd_{\mathcal{P}}$ on $\mathcal{P}$ as $Cd_{\mathcal{P}}([d], [d^{/}]) = \inf$ \{\, $Cd(d_{1}, d_{2})$ $\mid$  $d_{1}\in [d], d_{2}\in [d^{/}]$\,\}. Here, $[d], [d^{/}]\in\mathcal{P}$.
\end{definition}
\begin{lemma}
Let $B$ be the universe of big data, and DS be finite the document space of B, where DS $\subseteq B$. Let $\mathcal{P}$ be a partition on DS. Then, $< \mathcal{P}, Cd_{\mathcal{P}} >$ following properties: \begin{enumerate}
    \item $Cd_{\mathcal{P}}([d],[d^{/}])\geq 0,$
    \item $Cd_{\mathcal{P}}([d],[d^{/}]) = 0\iff [d] = [d^{/}]$,
    \item $Cd_{\mathcal{P}}([d],[d^{/}])=Cd_{\mathcal{P}}([d^{/}],[d])$
\end{enumerate}
\begin{proof}
\begin{enumerate}
	\item 	Let $d_{1}\in [d]$ and $d_{2}\in [d^{/}]$ be any two elements, where $[d],[d^{/}]\in \mathcal{P}.$\\ From lemma 3.1, we have $Cd(d_{1},d_{2})\geq 0.$ So, inf$\{\,Cd(d_{1},d_{2})\mid d_{1}\in [d],d_{2}\in[d^{/}]\,\}\geq 0.$ Thus, $Cd_{\mathcal{P}}([d],[d^{/}])\geq 0.$\\
	
\item  If possible let $[d]\neq[d^{/}].$ Then, for any $d_{1}\in [d],$ we have  $d_{1}\not\in [d^{/}].$ It implies that for any $d_{2}\in[d^{/}], $ we have $Cd(d_{1},d_{2})\neq0.$ Since, $d_{1}$ and $d_{2}$ are choosen arbitrarily so, we get  $Cd_{\mathcal{P}}([d],[d^{/}])\neq 0$. Thus, $[d]\neq[d^{/}]$ implies $Cd_{\mathcal{P}}([d],[d^{/}])\neq 0.$  Hence, $Cd_{\mathcal{P}}([d],[d^{/}])= 0$ implies $ [d]=[d^{/}].$

	Conversely, let $[d]$ = $[d^{/}]$. Now, we have 
	$Cd_{\mathcal{P}}([d],[d^{/}]) = \inf\{\, Cd(d_{1},d_{2})\mid d_{1}\in [d]$ and $d_{2}\in[d^{/}]$\,\}. Since, $[d]$ = $[d^{/}]$ so, $d_{1}\in [d^{/}]$ and thus, $d_{1}\approx d_{2}.$ So, $Cd(d_{1},d_{2})=0$ and thus, $Cd_{\mathcal{P}}([d],[d^{/}]) = \inf\{\, Cd(d_{1},d_{2})\mid d_{1}\in [d]$ and $d_{2}\in[d^{/}]$\,\}=0. 
	Hence, $Cd_{\mathcal{P}}([d],[d^{/}]) = 0\iff [d] = [d^{/}]$.\\

\item 	Again, $Cd_{\mathcal{P}}([d],[d^{/}])= \inf\{\, Cd(d_{1},d_{2})\mid d_{1}\in [d]$ and $d_{2}\in[d^{/}]\,\}=\inf\{\, Cd(d_{2},d_{1})\mid d_{1}\in [d]$ and $d_{2}\in[d^{/}]$\,\} = $Cd_{\mathcal{P}}([d^{/}],[d])$.\\

\end{enumerate}
\end{proof}

\end{lemma}
\begin{definition}
Let $B$ be the universe of big data, and QS and DS be query and document spaces, respectively, where DS $\subseteq B$.
For any query $q_{0}\in QS$, let there exist a document $d_{0}$ such that $q_{0}\approx d_{0}$. Then, the cognitive retrieval function $f:QS\longrightarrow \mathcal{P}$ such that $f(q_{0})=\{\,d\in DS\mid Cd(d_{0},d)$ $<\epsilon$\,\}=$[d_{0}]_{\epsilon}$, where $\epsilon\in (0,1)$. Also, the cognitive retrieval topology $\tau_{Cog}$ on DS can be defined as the topology generated by the set $\{\,f(q)\mid q\in $ QS$\,\}$ as the subbasis.
\end{definition}
\begin{remark}
Here, $d_{0}$ can be considered as a document that matches the cognitive expectation of a user for the query $q_{0}$, or simply we can say that $q_{0}\approx d_{0}$. The system ($B$, DS, QS, $f$, $\tau_{Cog}$) is called the cognitive retrieval system in big data.
\end{remark}
\begin{theorem}
In ($B$, DS, QS, $f$, $\tau_{Cog}$), for any two queries $q_{1}$ and $q_{2}$, their cognitive retrieval sets $f(q_{1})$ and $f(q_{2}) $ overlap with each other if $q_{1}\approx q_{2}$.\end{theorem}
\begin{proof}
To show that $f(q_{1})$ and $f(q_{2})$ overlap with each other if $q_{1} \approx q_{2}$, i.e., $f(q_{1})\cap f(q_{2})\neq\phi$ if $q_{1}\approx q_{2}$. Since $q_{1}\approx q_{2}$, so by definition 3.2 we have $Cd(q_{1},q_{2})=0. $ Also, consider that $d_{1}$ and $d_{2}$ are such documents that meet a user's cognitive expectations for the queries $q_{1}$ and $q_{2}$, respectively. Thus, $q_{1}\approx d_{1}$ and $q_{2}\approx d_{2}$. Then, again by using definition 3.2, we have that $Cd(d_{1},d_{2})=0. $ Let $d\in f(q_{1})$ be any document. Then, by definition 3.4, there exists $d_{1}\in$ DS with $d_{1}\approx q_{1}$ such that $Cd(d_{1},d)< \frac{\epsilon}{2}$, where $\epsilon\in(0,1)$. Now, by lemma 3.1, we have that $Cd(d_{2},d)\leq Cd(d_{2}, d_{1})+ Cd(d_{1},d)$. It implies that $Cd(d_{2},d)\leq Cd(d_{1}, d_{2})+ Cd(d_{1},d). $ So, $Cd(d_{2},d)\leq 0+\frac{\epsilon}{2}=\frac{\epsilon}{2}<\epsilon $. Therefore, $d\in f(q_{2})$. Hence, $f(q_{1})\cap f(q_{2})\neq\phi. $ 
\end{proof}
\begin{theorem}
In ($B$, DS, QS, $f$, $\tau_{Cog}$), the retrieval set for a fixed query is nested.
\end{theorem}
\begin{proof}
Let $q$ be a fixed query and $d_{0}$ be a document such that $q\approx d_{0}$. To show that for any $\epsilon_{1}, \epsilon_{2}\in(0,1)$, with $\epsilon_{1} <\epsilon_{2}$, we have that $[d_{0}]_{\epsilon_{1}}\subseteq [d_{0}]_{\epsilon_{2}}$. Let $d\in [d_{0}]_{\epsilon_{1}}$ be any document. Then, $Cd(d_{0},d)< \epsilon_{1}<\epsilon_{2}. $ This implies that $Cd(d_{0},d)< \epsilon_{2}$. So, $d\in [d_{0}]_{\epsilon_{2}}. $ Hence, $[d_{0}]_{\epsilon_{1}}\subseteq [d_{0}]_{\epsilon_{2}}. $
\end{proof}
In a cognitive retrieval system, users frequently interact with the system by refining their queries based on the feedback of the initial search results. Typically, users refine their queries by adding or removing words or by slightly modifying their intent, without significant disruptions in the cognitive behaviour between the queries. Let $q$ be a query and $q^{/}$ be a perturbed query from $q$. Here, $q^{/}$ can be defined as $q^{/}=q\oplus\delta q$, where $\delta q$ is the small changes that are made to $q$ so that there do not exist any significant cognitive changes between $q$ and $q^{/}$; here $\oplus$ symbols are used to define perturbation. Mathematically, a query $q$ can be perturbed to $q^{/}$ without any significant cognitive change if for each $\delta q$ there exists $\epsilon \in (0,1)$ such that $Cd(q,q^{/})< \epsilon$ with $q^{/}=q\oplus\delta q$. For instance, consider a user searching for research papers on `machine learning in healthcare'.  If the user slightly alters the query to `deep learning in healthcare'; the system should still return many relevant documents, particularly if those documents contain content related to both machine learning and deep learning. This cognitive stability between the retrieved sets is crucial; otherwise, small changes in queries could lead to large shifts in results, making the system feel unpredictable and frustrating for users. Thus, we discuss some stability results based on this idea below:
\begin{definition}
Let $q$ be any query which is perturbed to a new query $q^{/}$. Then, their respective retrieval sets will be cognitively stable if the distance, $Cd_{\mathcal{P}}$, between their respective retrieval sets is zero.
\end{definition}
\begin{theorem}
In ($B$, DS, QS, $f$, $\tau_{Cog}$), if query $q$ is slightly perturbed to a new query $q^{/}$ without any significant cognitive change, then their retrieval sets are cognitively stable.
\end{theorem}
\begin{proof}
Let us consider that there exist documents $d$ and $d^{/}$ such that $q\approx d$ and $q^{/}\approx d^{/}$. Then, by definition 3.2 , we have $Cd(q,d)=0$ and $Cd(q^{/},d^{/})=0$. Since, without any significant cognitive changes, $q$ is perturbed to $q^{/}$. Then, there exists $\epsilon\in (0,1)$ such that $Cd(q,q^{/})< \frac{\epsilon}{3}. $ To show that, retrieval sets $[d]_{\frac{\epsilon}{3}} $ and $[d^{/}]_{\frac{\epsilon}{3}}$ are cognitively stable. Again from definition 3.3, we have that $Cd_{\mathcal{P}}([d]_{\frac{\epsilon}{3}},[d^{/}]_{\frac{\epsilon}{3}})= \inf$\{\, $Cd(d_{1},d_{2})$ $\mid$  $d_{1}\in [d]_{\frac{\epsilon}{3}}$ and $d_{2}\in [d^{/}]_{\frac{\epsilon}{3}}$\,\}.  But $d_{1}\in [d]_{\frac{\epsilon}{3}}$ implies that $Cd(d_{1},d)<{\frac{\epsilon}{3}} $ and similarly we have that $Cd(d_{2},d^{/})<{\frac{\epsilon}{3}}.$ Now, from lemma 3.1, we have that $Cd(d_{1},d_{2})\leq Cd(d_{1},d)+Cd(d,q)+Cd(q,q^{/})+Cd(q^{/},d^{/})+Cd(d^{/},d_{2}). $ This implies that $Cd(d_{1},d_{2})\leq Cd(d_{1},d)+Cd(q,d)+Cd(q,q^{/})+Cd(q^{/},d^{/}) +Cd(d_{2},d^{/})$. So, $Cd(d_{1},d_{2})\leq{\frac{\epsilon}{3}}+0+\frac{\epsilon}{3}+0+{\frac{\epsilon}{3}} $. Thus, we have that $Cd(d_{1},d_{2})\leq {\epsilon}.$ So, we get $Cd_{\mathcal{P}}([d]_{\frac{\epsilon}{3}},[d^{/}]_{\frac{\epsilon}{3}})=0.$ Hence, by definition 3.5, they are cognitively stable. 
\end{proof}

Now, we can think of a sequence of queries $\{\,q_{n}\,\}_{n=1}^{\infty}$ in QS obtained by perturbing $q_{1}$ without any cognitive changes. Where, $q_{n}=q_{n-1}\oplus\delta q_{n-1}, \forall n=2,3,...$ .Then, we will also have a sequence of retrieval sets of each of the terms in $\{\,q_{n}\,\}_{n=1}^{\infty}$. Since the sequence $\{\, q_{n}\,\}_{n=1}^{\infty}$ is obtained by perturbing $q_{1}$ in a recurrent manner and this perturbation is done to get a desired query, say $q$, so the sequence cannot be divergent in general. Again, the sequence $\{\,q_{n}\,\}$ cognitively converges to $q$ if for each $\epsilon\in (0,1),$ there exists $m\in\mathcal{N}$, such that $Cd(q_{n},q)< \epsilon$, whenever $n\geq m.$ In large-scale retrieval system, specifically for retrieval in big data, achieving stability is critical as it ensures that minor perturbations in input do not cause instability in output. The following theorem will help to handle such changes smoothly, which is essential for real-time and large-scale applications. 

\begin{theorem}
In ($B$, DS, QS, $f$, $\tau_{Cog}$), if the sequence of queries $\{\,q_{n}\,\}_{n=1}^{\infty}$ where $q_{n}=q_{n-1}\oplus\delta q_{n-1},\forall n=2,3,...$, is cognitively convergent, then the retrieval sets of $\{\,q_{n}\,\}_{n=1}^{\infty}$ are cognitively stable. 
\end{theorem}
\begin{proof}
Let the sequence $\{\,q_{n}\,\}$ converge to $q$. Then, for each $\epsilon\in(0,1),$ there exists $m\in \mathcal{N}$ such that $Cd(q_{n},q)<\epsilon,\forall n\geq m.$ Also, for each $n, [d_{n}]_{\epsilon_{n}}$ be retrieval sets for $q_{n}$, where $d_{n}\in $ DS, with $d_{n}\approx q_{n}$, for each $n$.  Here, $Cd(q_{n},q)<\epsilon$ implies that $q_{n}$ is perturbed to $q$ without any cognitive changes. So from theorem 3.3, we can assert that their retrieval sets are cognitively stable.
\end{proof}
Sometimes, a user can find some unusual information for a query in the retrieval system. In unusual way, it means that the information that is may not be the user's cognitive expectation, although the information is in the cognitive retrieval set for the query. To solve this, context-based cognitive retrieval methods can be used. The context in cognitive information retrieval may refer to any additional information beyond the query itself, such as user's intent, search history, geographic location, device type, or even implicit factors like recent interactions, etc. Now, we can see how the cognitive similarity distance works here. For example, in commercial or transactional contexts, the words `$Buy$' and `$Purchase$' have a similar meaning. So, we can say that $Cog(Buy,Purchase)= 0$. That is, we can write as for the transactional context $Buy\approx Purchase.$ Mathematically, for a context $c$, if for two words $x,y$ we have $Cog(x,y)=0$, then we write it as $x\approx_{c}y.$  Now, we see how our cognitive retrieval function works here.

\begin{definition}
Let $C$ be the set of contexts. In ($B$, DS, QS, $f$, $\tau_{Cog}$), context-based cognitive retrieval for a query $q_{0}$ can be defined as $f_{c}(q_{0})=\{\,d\in$ DS $\mid Cd(d_{0},d)< \epsilon $ and $ q_{0}\approx_{c}d_{0}\,\}$, for $\epsilon\in (0,1)$ and $c\in C$. The system ($B$, DS, QS, $C$, $f_{c}$, $\tau_{Cog}$) is called a context-based cognitive retrieval system in big data. Also, any subset $D\in$ DS is cognitively connected if there does not exist any disjoint open set $G$ and $H$ in $\tau_{Cog}$ such that $D= G\cup H$. By disjoint it means that, there does not exist any $d\in G, d^{/}\in H$ such that $d\approx_{c}d^{/}.$ 
\end{definition}

\begin{theorem}
In ($B$, DS, QS, $C$, $f$, $\tau_{Cog}$), the cognitive retrieval set $f_{c}(q_{0})$ is cognitively connected, where $c\in C$ and $q_{0}\in$QS.
\end{theorem}
\begin{proof}
If possible, let $f_{c}(q_{0})$ be cognitively disconnected. Then, there exist disconnections $G$ and $H$ in $\tau_{Cog}$ such that $f_{c}(q_{0})= G\cup H$.  Let $d_{1}\in G$ and $d_{2}\in H$ be any documents. Then, $d_{1}, d_{2}\in f_{c}(q_{0}). $ Then, by definition 2.6, for a context $c$, there exists $d_{0}\in$ QS such that $Cd(d_{0}, d_{1})< \frac{\epsilon}{2}$ and $Cd(d_{0},d_{2})<\frac{\epsilon}{2}$.  Now, we have $Cd(d_{1}, d_{2})\leq \frac{\epsilon}{2}+\frac{\epsilon}{2}=\epsilon. $ But $\epsilon\in (0,1)$ is taken arbitrarily, so $Cd(d_{1}, d_{2})=0. $ Thus, $d_{1}\approx_{c} d_{2},$ but it is not possible because $G$ and $H$ are disjoint. So, $f_{c}(q_{0})$ is cognitively connected.
\end{proof}

\section{Cognitive information retrieval in big data using logical connectives:}
Boolean algebra plays a crucial role in cognitive information retrieval (CIR) because it provides a foundational framework for structuring and manipulating queries in ways that align with how users think and search for information. It allows users to construct complex queries by combining simple terms using logical connectives such as `AND', `OR', and `NOT'. Here, connective `AND' retrieves documents that must contain all specified terms (e.g., `machine learning' and `health care'). Connective `OR' retrieves documents that contain at least one of the specified terms (e.g., `machine learning' or `artificial intelligence'). Also, connective `NOT' excludes documents containing certain terms for example, if the query `machine learning NOT finance' is submitted, then the retrieval system starts excluding all documents that contain the term `finance'. Cognitive retrieval systems often deal with multiple dimensions of user queries combining keywords, semantic information, and user contexts, where mathematical logic allows for the integration of multiple cognitive dimensions, helping the system to weigh different factors and return documents that meet the combined criteria. \\

\subsection{ Connective `AND' }

\noindent
If queries $q_{1}, q_{2}$ are joined by the connective `AND' , then symbolically it is written as $q_{1}\wedge q_{2}.$ Also, let $d_{0}$ be a document such that $d_{0}\approx q_{1}$ and $d_{0}\approx q_{2}$, then we write $d_{0}\approx q_{1}\wedge q_{2}.$ Thus, in ($B$, DS, QS, $f$, $\tau_{Cog}$), the cognitive retrieval set for $q_{1}\wedge q_{2}$ can be defined as $f(q_{1}\wedge q_{2})=\{\,d\in$ DS $\mid Cd(d_{0},d)<\epsilon\,\},$ where $\epsilon\in (0,1)$. Hence, it is noticeable that if two queries are joined by connective AND, then the retrieval set for the new query will be the overlap of the retrieval sets for those two queries. The following theorem is based on this idea:

\begin{theorem}

In ($B$, DS, QS, $f$, $\tau_{Cog}$), we have that $f(q_{1}\wedge q_{2})=f(q_{1})\cap f(q_{2}). $

\end{theorem}

\begin{proof}

Let $d\in$ DS such that $d\in f(q_{1}\wedge q_{2}). $ Then, there exists a document $d_{0}$ with $d_{0}\approx q_{1}\wedge q_{2}$ such that $Cd(d,d_{0})<\epsilon$.  But $d_{0}\approx q_{1}\wedge q_{2}$ implies that $d_{0}\approx q_{1}$ and $d_{0}\approx q_{2}$.  Thus, $d_{0}\approx q_{1}$ implies that $Cd(d,d_{0})<\epsilon$ and so, $d\in f(q_{1}). $ Similarly, we can show that $d\in f(q_{2}). $ Thus, $d\in f(q_{1})\cap f(q_{2})$ and hence, $f(q_{1}\wedge q_{2})\subseteq f(q_{1})\cap f(q_{2}). $\\

Conversely, assume that $d\in f(q_{1})\cap f(q_{2}).$ Then, $d\in f(q_{1})$ and $d\in f(q_{2}). $ By definition 3.4, we have that there exist $d_{1}$ and $d_{2}$ with $q_{1}\approx d_{1}$ and $q_{2}\approx d_{2}$ such that $Cd(d,d_{1})<\epsilon$ and $Cd(d,d_{2})<\epsilon$.  Again, let $d_{0}\approx q_{1}\wedge q_{2}.$ Then, we have that $d_{0}\approx q_{1}$ and $q_{1}\approx d_{1}$ imply that $d_{0}\approx d_{1}$. Hence, $Cd(d_{0},d_{1})=0.$ Now, we have $Cd(d,d_{0})<Cd(d,d_{1})+Cd(d_{1},d_{0})$. So, $Cd(d,d_{0})<\epsilon+0=\epsilon.$ Therefore, $d\in f(q_{1}\wedge q_{2}).$ Hence, $f(q_{1}\wedge q_{2})\subseteq f(q_{1})\cap f(q_{2}).$ Thus, $f(q_{1}\wedge q_{2})= f(q_{1})\cap f(q_{2}). $

\end{proof}

\begin{corollary}

In ($B$, DS, QS, $f$, $\tau_{Cog}$), $f(q_{1}\wedge q_{2})=\{\, d\in$ DS $\mid Cd(d,d_{1})<\epsilon_{1}$ and $Cd(d,d_{2})<\epsilon_{2}\,\}$, where query $q_{1}\approx d_{1}$ and $q_{2}\approx d_{2}$. Here, $\epsilon_{i}\in(0,1),i=1,2.$

\end{corollary}

\begin{proof}

Let $d\in$ DS be any document such that $d\in f(q_{1}\wedge q_{2}). $ But from theorem 4.1, we have that $d\in f(q_{1})\cap f(q_{2}).$ This implies $d\in f(q_{1})$ and $d\in f(q_{2}).$ Then, $Cd(d,d_{1})<\epsilon_{1}$ and $Cd(d,d_{2})<\epsilon_{2}$. So, $f(q_{1}\wedge q_{2})\subseteq\{\, d\in$ DS $\mid Cd(d,d_{1})<\epsilon_{1}$ and $Cd(d,d_{2})<\epsilon_{2}\,\} $.\\ 

Conversely, let $d\in$ DS such that $Cd(d,d_{1})<\epsilon_{1}$ and $Cd(d,d_{2})<\epsilon_{2}$. Then, $d\in f(q_{1})$ and $d\in f(q_{2})$. So, $d\in f(q_{1})\cap f(q_{2}). $ Thus, $f(q_{1}\wedge q_{2})\supseteq\{\, d\in$ DS $\mid Cd(d,d_{1})<\epsilon_{1}$ and $Cd(d,d_{2})<\epsilon_{2}\,\}$. Hence, the result is proved.
\end{proof}

\subsection{ Connective `OR' }

\noindent
Suppose two queries `machine learning', `big data' are joined by the `OR' connective to form a new query `machine learning OR big data', then retrieval system retrieves information about either `machine learning' or `big data'. If query $q_{1},q_{2}$ are joined by the operator `OR', then symbolically it is written as $q_{1}\vee q_{2}.$ Also, let $d_{0}$ be a document such that $d_{0}\approx q_{1}$ and $d_{0}\approx q_{2}$, then we write $d_{0}\approx q_{1}\vee q_{2}.$ Thus, in ($B$, DS, QS, $f$, $\tau_{Cog}$), the cognitive retrieval set for $q_{1}\vee q_{2}$ can be defined as $f(q_{1}\vee q_{2})=\{\,d\in$ DS $\mid Cd(d_{0},d)<\epsilon\,\}$, where $\epsilon\in (0,1)$.\\

\begin{theorem}
In ($B$, DS, QS, $f$, $\tau_{Cog}$),  $f(q_{1}\vee q_{2})=f(q_{1})\cup f(q_{2}).$
\end{theorem}
\begin{proof}
Let $d\in f(q_{1}\vee q_{2})$ be any document. Then, there exists $d_{0}\in$ DS such that $d_{0}\approx q_{1}\vee q_{2}$ with $Cd(d_{0},d)< \epsilon.$ But, $d_{0}\approx q_{1}\vee q_{2}$ implies that $d_{0}\approx q_{1} $ or $d_{0}\approx q_{2}$. So, there exist $d_{0}\in$ DS such that $Cd(d_{0},d)<\epsilon$ or $Cd(d_{0},d)<\epsilon.$ Therefore, $d\in f(q_{1})$ or $d\in f(q_{2})$ and so, $d\in f(q_{1})\cup f(q_{2}).$ Hence, $f(q_{1}\vee q_{2})\subseteq f(q_{1})\cup f(q_{2}).$\\

Conversely, let $d \in DS$ such that $d\in f(q_{1})\cup f(q_{2}).$  Then, $d\in f(q_{1})$ or $ d\in f(q_{2}).$ 
If $d\in f(q_{1})$ then, there exists $d_{1}\in$ DS with $q_{1}\approx d_{1}$  such that $Cd(d,d_{1})< \epsilon$. Now, let $d_{0}\approx q_{1}\vee q_{2}$ and so, $Cd(d_{1},d_{0})=0$. Then, we have $Cd(d,d_{0})\leq Cd(d,d_{1})+ Cd(d_{1}, d_{0}).$ So, $Cd(d,d_{0})< \epsilon.$ Thus, $d\in f(q_{1}\vee q_{2})$ and hence, $f(q_{1}\vee q_{2})\subseteq f(q_{1})\cup f(q_{2}).$ Thus, $ f(q_{1}\vee q_{2})=f(q_{1})\cup f(q_{2}).$
\end{proof}
\begin{corollary}
In ($B$, DS, QS, $f$, $\tau_{Cog}$),  we have $f(q_{1}\vee q_{2})=\{\, d\in$ DS $\mid Cd(d,d_{1})<\epsilon_{1}$ or $Cd(d,d_{2})<\epsilon_{2}\,\}$, where query $q_{1}\approx d_{1}$ or $q_{2}\approx d_{2}.$  
\end{corollary}
\begin{proof}
Let $d\in$ DS such that $d\in f(q_{1}\vee q_{2}).$ By theorem 4.2, we have that $d\in f(q_{1})\cup f(q_{2}). $ Then, $d\in f(q_{1})$ or $d\in f(q_{2}).$ If $d\in f(q_{1})$, then there exists $d_{1}\in$ DS with $q_{1}\approx d_{1}$ such that $Cd(d,d_{1})<\epsilon_{1}$, where $\epsilon_{1}\in (0,1).$ Similarly, we can show that if $d\in f(q_{2})$, then $Cd(d,d_{2})<\epsilon_{2},$ where $\epsilon_{2}\in (0,1).$ Thus,  $f(q_{1}\vee q_{2})\subseteq\{\, d\in$ DS $\mid Cd(d,d_{1})<\epsilon_{1}$ or $Cd(d,d_{2})<\epsilon_{2}\,\}$. The other part is easy to show that $f(q_{1}\vee q_{2})\supseteq\{\, d\in$ DS $\mid Cd(d,d_{1})<\epsilon_{1}$ or $Cd(d,d_{2})<\epsilon_{2}\,\}$. Hence,  $f(q_{1}\vee q_{2})=\{\, d\in$ DS $\mid Cd(d,d_{1})<\epsilon_{1}$ or $Cd(d,d_{2})<\epsilon_{2}\,\}$.
\end{proof}

\subsection{Connective `NOT'} 

\noindent
Suppose a user makes a query as `price of shoes excluding the riding boots' in a database. Then, the retrieval system will obviously exclude the term  `riding boots' from the retrieval process. Here, we can use the `NOT' connective. Let $q$ be the original given query and we parse it into some sub-queries $q_{1}=$ `price of shoes' and $q_{2}=$ `riding boots'. Then mathematically, query $q$ can be written as $q=q_{1}\wedge  \leftharpoondown  q_{2}$. Here, $`\leftharpoondown  q_{2}$' means `NOT $q_{2}$'. Now, let $d_{0}$ be a document in DS such that $d_{0}\approx q_{2}$ . In ($B$, DS, QS, $f$, $\tau_{Cog}$), the cognitive retrieval set for $\leftharpoondown  q_{2}$ can be defined as $f(\leftharpoondown  q_{2})=\{\,d\in$ DS $\mid Cd(d_{0},d)\geq\epsilon\,\},$ where $\epsilon\in (0,1)$.
\begin{theorem}
In ($B$, DS, QS, $f$, $\tau_{Cog}$), we have that $f(\leftharpoondown  q)= $ (DS$-f(q))$.  
\end{theorem}
\begin{proof}
It is known that $f(\leftharpoondown  q)=\{\,d\in$ DS $\mid Cd(d_{0},d)\geq\epsilon\,\},$ where $d_{0}\approx q_{0}$ and $\epsilon\in(0,1)$. Let $d\in f(\leftharpoondown  q)$ iff $Cd(d_{0},d)\geq\epsilon$ iff $d\notin f(q)$ iff $d\in$ DS$-f(q).$ Hence, $f(\leftharpoondown  q)= $ (DS$-f(q))$.
\end{proof}

In the above part, it has shown how this set of connectives $\{\, \vee, \wedge,\leftharpoondown  \,\}$ is used to cognitively retrieve any Boolean queries, yet there are some cases where connectives like implication ($\rightarrow$) and bi-implication ($\leftrightarrow$) become essential to form a Boolean query. For example, suppose it is asked to retrieve all the documents where if it is about `machine learning', then also be about `artificial intelligence'. Then, we form a Boolean query, $q_{1}\rightarrow q_{2}$, where $q_{1}$ is about documents of machine learning, and $q_{2}$ is a query about documents of artificial intelligence. In the following, it is shown how this query `$q_{1}\rightarrow q_{2}$' can be retrieved using the set of connectives, $\{\,\vee,\leftharpoondown \,\}$. \\
\begin{theorem}
In ($B$, DS, QS, $f$, $\tau_{Cog}$), for any $q_{1}, q_{2}\in$ QS, $f(q_{1}\rightarrow q_{2})=$ (DS$-f(q_{1}))\cup f(q_{2}).$
\end{theorem}
\begin{proof}
It is obvious that, $q_{1}\rightarrow q_{2}\equiv (\leftharpoondown q_{1})\vee q_{2}.$ Thus, $f(q_{1}\rightarrow q_{2})=f((\leftharpoondown q_{1})\vee q_{2})$. Again, by theorems 4.2, and 4.3 we have $f((\leftharpoondown q_{1})\vee q_{2})=$(DS$-f(q_{1}))\cup f(q_{2}).$
\end{proof}

The above part of this section is dedicated to basic properties and notions of logical connectives from the perspective of cognitive retrieval in a big data environment. In the following part, we are going to explain that $\tau_{Cog}$ can be considered as the collection of all possible Boolean retrievals based on queries of QS.

Let $(q_{ij})$, where $i=1,2,...,m$ and $j=1,2,...,n$, be array of queries in QS. Then Boolean query can be denoted as $\vee_{j=1}^{n}(\wedge_{i=1}^{m}q_{ij})$ . 
\begin{lemma}
In ($B$, DS, QS, $f$, $\tau_{Cog}$), for Boolean query $\vee_{j=1}^{n}(\wedge_{i=1}^{m}q_{ij})$, where $(q_{ij})\in$ QS and $i=1,2,...,m$ and $j=1,2,...,n$, we get $f(\vee_{j=1}^{n}(\wedge_{i=1}^{m}q_{ij}))= \cup_{j=1}^{n}(\cap_{i=1}^{m}f(q_{ij}))$.
\end{lemma}
\begin{proof}
Here, $\vee_{j=1}^{n}(\wedge_{i=1}^{m}q_{ij})=(q_{11}\wedge q_{21}\wedge...q_{m1})\vee(q_{12}\wedge q_{22}\wedge...q_{m2})\vee...\vee (q_{1n}\wedge q_{2n}\wedge...q_{mn}).$ Using theorems 4.1 and 4.2, we have $f(\vee_{j=1}^{n}(\wedge_{i=1}^{m}q_{ij}))=\cup_{j=1}^{n}f(q_{1j}\wedge q_{2j}\wedge...q_{mj})$ = $\cup_{j=1}^{n}(\cap_{i=1}^{m}f(q_{ij})).$
\end{proof} The following theorem shows that  $\tau_{Cog}$ can be considered as a set of all Boolean retrievals.
\noindent
\begin{theorem}\label{teopr}
In ($B$, DS, QS, $f$, $\tau_{Cog}$), $\tau_{Cog}$ is equal to the set of all possible Boolean cognitive retrievals of elementary queries of QS.
\end{theorem}
\begin{proof}
Let $(q_{ij})$, where $i=1,2,...,m$ and $j=1,2,...,n$, be array of queries in QS. Then, Boolean query can be denoted as $\vee_{j=1}^{n}(\wedge_{i=1}^{m}q_{ij})$. Also, from lemma 4.1, we have $f(\vee_{j=1}^{n}(\wedge_{i=1}^{m}q_{ij}))= \cup_{j=1}^{n}(\cap_{i=1}^{m}f(q_{ij}))$. Again, from definition 3.4, we also have that $\tau_{Cog}$ is generated by $\{\,f(q)\mid q\in $ QS$\,\}$ as sub basis and it is clear that for any $i,j$,  $f(q_{ij})\in$ $\{\,f(q)\mid q\in $ QS$\,\}$ .Thus, the set $G=\cup_{j=1}^{n}(\cap_{i=1}^{m}f(q_{ij}))$ is a open set of $\tau_{Cog}$. Hence,  $\tau_{Cog}$ is the  set of all possible Boolean cognitive retrievals of elementary queries of QS.
\end{proof}
It is well-known that  $\{\,\vee, \leftharpoondown \,\}$, $\{\,\wedge, \leftharpoondown\,\}$ and $\{\, \vee, \wedge, \leftharpoondown\,\}$ form adequate system of connectives \cite{23}. Thus, we have the following theorem for retrieval function $f$: 
\begin{theorem}
In ($B$, DS, QS, $f$, $\tau_{Cog}$), for any $p,q\in$ QS, \begin{enumerate}
	\item $f(p\vee q)=f(\leftharpoondown(\leftharpoondown p\wedge \leftharpoondown q)),$
	\item $f(p\wedge q)=f(\leftharpoondown(\leftharpoondown p\vee\leftharpoondown q)),$
	
	\item $f(p\leftrightarrow q)=f(p\rightarrow q)\cap f(p \leftarrow q).$
	
\end{enumerate} 
\end{theorem}
\begin{proof}
\begin{enumerate}
	\item  Using theorem 4.1 and  4.3, we have that $f(\leftharpoondown(\leftharpoondown p\wedge \leftharpoondown q))=$ DS$-(f(\leftharpoondown p\wedge \leftharpoondown q))=$ DS$-(f(\leftharpoondown p)\cap f(\leftharpoondown q))=$ DS$-(($DS$-f(p)\cap$(DS$-f(q)$)$=f(p)\cup f(q)=f(p\vee q).$
	\item Using theorem 4.2 and 4.3, We have $f(\leftharpoondown(\leftharpoondown p\vee \leftharpoondown q))=$ DS$-(f(\leftharpoondown p\vee \leftharpoondown q))=$ DS$-(f(\leftharpoondown p)\cup f(\leftharpoondown q))=$ DS$-(($DS$-f(p)\cup$(DS$-f(q)$)$=f(p)\cap f(q)=f(p\wedge q).$
	
	\item  We know that $f(p\leftrightarrow q)=f((p\rightarrow q)\vee(q\rightarrow p))=f(s\vee r)=f(s)\cup f(r)=f(p\rightarrow q)\cup f(q\rightarrow p),$ where $r=p\rightarrow q$ and $s=q\rightarrow p.$

\end{enumerate}
\end{proof}
	The above result ensures us that cognitive retrieval can be possible for any complex query using logical connectives.\\
    \section{ Application of stability for information retrieval under perturbation of queries:}
 Many of the above results on the stability of information retrieval under query perturbations have significant practical applications. A particularly critical domain is defense and intelligence operations, where query refinement is both frequent and mission-critical\cite{34}. In such operational environments, analysts typically begin with an initial query and iteratively adjust it in response to preliminary search results, situational updates, or feedback from intelligence officers. Maintaining retrieval stability under these perturbations is vital for operational consistency and reliable decision support. In military intelligence, even minor lexical modifications can otherwise trigger disproportionate shifts in the retrieved dataset, leading to gaps or distortions in the intelligence picture. To illustrate this challenge, we consider the following example.

 \begin{example}
     An intelligence expert queries a classified defense big data for $q$ = ``enemy drone activity in Northern Sector''. After preliminary review, the analyst changes the search term drone to UAV to align with mission-specific terminology, resulting in: $q'$ = ``enemy UAV activities in Northern Sector''.\\

Here, \( q' = q \oplus \delta q \), where \( \delta q \) represents the cognitive change of meaning while replacing the word  ``drone'' with the word``UAV''. From a cognitive similarity perspective, drone and UAV are considered cognitively similar \cite{35}, we have $Cog($drone, UAV) = 0. Thus, $\delta{q}$ approaches to zero. Thus,  $q\approx q^{/}$. Thus, by Theorem~3.3, the retrieval sets $f(q)$ and $f(q')$ remain cognitively stable, i.e., $Cd_{\mathcal{P}}(f(q), f(q^{'})) = 0. $ \\

This ensures that both queries yield nearly identical intelligence reports, preserving operational continuity and avoiding loss of critical situational awareness, we add the following  pseudo algorithm:\\

\begin{algorithm}[H]
\caption{Perturbation-Stable Cognitive Retrieval for Defense Intelligence}
\KwIn{Initial query $q$, perturbed query $q' = q \oplus \delta q$, threshold $\epsilon \in (0,1)$}
\KwOut{Cognitively stable retrieval set $R$}

\textbf{Step 1:} Identify cognitive representative documents\\
\hspace*{1em}Find $d$ such that $q \approx d$\\
\hspace*{1em}Find $d'$ such that $q' \approx d'$

\textbf{Step 2:} Check perturbation constraint\\
\hspace*{1em}If $Cd(q, q') < \frac{\epsilon}{3}$, continue; else flag as significant cognitive change

\textbf{Step 3:} Compute retrieval sets\\
\hspace*{1em}$R_q \gets \{ x \in DS \mid Cd(d, x) < \epsilon/3 \}$\\
\hspace*{1em}$R_{q'} \gets \{ x \in DS \mid Cd(d', x) < \epsilon/3 \}$

\textbf{Step 4:} Stability check\\
\hspace*{1em}If $Cd_\mathcal{P}(R_q, R_{q'}) = 0$:\\
\hspace*{2em}Return $R \gets R_q \cup R_{q'}$ (stable set)\\
\hspace*{1em}Else:\\
\hspace*{2em}Return \emph{Warning: Cognitive drift detected}

\end{algorithm}
 \end{example}
    
\section{Discussion:}
This paper introduces a topology-based method for cognitive information retrieval (CIR) in big data environments by integrating cognitive distance between users' queries and stored documents in a database; further logical connectives are used to enhance this method for vast numbers of complex queries. Each theorem in this paper represents a key contribution to the development of the CIR framework in big data. We find in the beginning of section 3 that document space DS is divided into some partitions, which is the key part of distributed computing for big data. As we move forward in section 3, theorem 3.1 is about the overlap of retrieval sets, which ensures that semantically related queries will produce consistent results. Also, the adaptability of the cognitive retrieval function that adapts to different levels of query precision can be seen in theorem 3.2. It also ensures that more specific queries refine the search space without losing previous results. In this paper, theorem 3.3 has key importance as it introduces cognitive stability and proves that small perturbations in queries do not disrupt the cognitive structure of retrieval sets. This result confirms the robustness of our methods to minor changes in query formulation, which is essential for real-world applications where user intent may vary slightly. As we move further in this paper, we attain it at section 3. This section mainly deals with CIR using logical connectives like `OR',AND', `NOT', etc. Though Boolean retrieval was one of the earliest methods of information retrieval, it was probably the most criticized method, as its disadvantage was that it either retrieves a document or not and does not provide a ranking of retrieved documents \cite{3}. In this paper, it is tried to overcome, to some extent, this ranking issue of Boolean methods by incorporating it with our cognitive distance, as in our methods, as cognitive distance measures how similar a document is to a query in terms of meaning, not just keyword matching. In the future, our methods will help to rank documents more effectively by how well they align with the user's intent. Further, Theorem 4.5 ensures that the cognitive retrieval topology $\tau_{Cog}$ corresponds to the set of all possible Boolean cognitive retrievals of elementary queries. This result provides a complete topological foundation for CIR, enabling systematic handling of logical query structures. Lastly, incorporating an adequate system of connectives in CIR, Theorem 4.6 ensures that any complex query can be processed within the CIR framework, highlighting the expressiveness of the proposed system.\\
\section{Conclusion:}
This paper upgraded the topology-based method of IR by incorporating cognitive theory and logical connectives. The proposed method enables a more nuanced and human-like understanding of user queries. Including the definitions of cognitive distance, cognitive retrieval function, and results related to the cognitive stability  are the key theoretical contributions of this paper. The establishment of cognitive stability and convergence provides a solid foundation for future advancements in CIR. The theorems on CIR using logical connectives and cognitive distance demonstrate how logical operators, viz; AND, OR, NOT, etc., can be extended to handle cognitive similarities, offering more precise and context-aware search outcomes. Also, using logical connectives ensures the proposed method can handle complex and nuanced user queries with greater accuracy. The stability and convergence properties established by the theorems ensure that the system maintains consistency even under dynamic query conditions. This makes the framework particularly suitable for large-scale applications such as search engines, scientific literature databases, and commercial information systems.\\

Although cognitive distance improves the semantic relevance of retrieved documents, it does not fully address the ranking issue inherent to Boolean models. Future work could explore hybrid models that combine cognitive distance with scoring-based ranking mechanisms or machine learning approaches to improve the ordering of results. Furthermore, by extending the framework to handle multimodal data and real-time feedback systems, it would further improve its adaptability and performance in complex big data environments.\\

Thus, this paper establishes a solid theoretical foundation CIR in big data environments by bridging the gap between human cognitive processes and computational efficiency and paving the way for more intelligent and context-aware information retrieval systems.

\vskip 1cm
\noindent
{\bf Authors' contributions:} S.A. and R.C. conceptualized the idea, S.A. and R.C. wrote the manuscript, R.C. used software, S.A. supervised.  All authors reviewed the manuscript.\\
\noindent
{\bf Ethics.} This work did not require ethical approval from a human subject or animal welfare committee.\\
{\bf Data availability.} No datasets were generated or analysed during the current study.
.\\
{\bf Declaration of AI use.} We have not used AI-assisted technologies in creating this article.\\
{\bf Conflict of interest declaration.} We declare we have no competing interests.\\
{\bf Funding.} No funding has been received for this article.\\

\makeatletter

\end{document}